\documentclass[12pt]{amsart}
\usepackage{amssymb}

\newtheorem{theorem}{Theorem}[section]
\newtheorem{lemma}[theorem]{Lemma}

\numberwithin{equation}{section}

\begin{document}

\title[A Hasse-type principle and its applications]{A Hasse-type principle for exponential diophantine equations and its applications}

\author{Csan\'ad Bert\'ok}
\address{Institute of Mathematics\\
University of Debrecen\\
H-4010 Debrecen, P.O. Box 12\\
Hungary}
\email{bertok.csanad@gmail.com}

\author{Lajos Hajdu}
\address{Institute of Mathematics\\
University of Debrecen\\
H-4010 Debrecen, P.O. Box 12\\
Hungary}
\email{hajdul@science.unideb.hu}
\thanks{Research supported in part by the OTKA grants K100339 and NK101680, and by the T\'AMOP-4.2.2.C-11/1/KONV-2012-0001 project. The project has been supported by the European Union, co-financed by the European Social Fund.}



\date{}

\begin{abstract}
We propose a conjecture, similar to Skolem's conjecture, on a Hasse-type principle for exponential diophantine equations. We prove that in a sense the principle is valid for "almost all" equations. Based upon this we propose a general method for the solution of exponential diophantine equations. Using a generalization of a result of Erd\H{o}s, Pomerance and Schmutz concerning Carmichael's $\lambda$ function, we can make our search systematic for certain moduli needed in the method.
\end{abstract}

\maketitle

\section{Introduction}

Let $a_1,\dots,a_k$, $b_{11},\dots,b_{1\ell},\dots,b_{k1},\dots,b_{k\ell}$ be non-zero integers, $c$ be an integer, and consider the exponential diophantine equation
\begin{equation}
\label{eq1}
a_1 b_{11}^{\alpha_{11}}\dots b_{1\ell}^{\alpha_{1\ell}}+\dots+
a_k b_{k1}^{\alpha_{k1}}\dots b_{k\ell}^{\alpha_{k\ell}}=c
\end{equation}
in non-negative integers $\alpha_{11},\dots,\alpha_{1\ell},\dots,
\alpha_{k1},\dots,\alpha_{k\ell}$.

The effective and ineffective theory of \eqref{eq1} has a long history. In case of $k=2$, one can apply Baker's method to give explicit bounds for the exponents $\alpha_{11},\dots,\alpha_{1\ell}, \alpha_{21},\dots,\alpha_{2\ell}$; see e.g. results of Gy\H{o}ry \cite{gy1,gy2}. Note that by results of Vojta \cite{v} and Bennett \cite{be}, the solutions to \eqref{eq1} can still be "effectively determined" for $k=3,4$, under some further restrictive assumptions. On the other hand, it is also known that for any $k$, the number of those solutions to equation \eqref{eq1} for which the left hand side of has no vanishing subsum is finite, and it can be bounded explicitly in terms of $k$ and $\ell$ (see \cite{ess} and \cite{av}, and the references given there).

In this paper we propose the following

\vskip.2truecm

\noindent
{\bf Conjecture.} Suppose that equation \eqref{eq1} has no solutions. Then there exists an integer $m$ with $m\geq 2$ such that the congruence
\begin{equation}
\label{eq2}
a_1 b_{11}^{\alpha_{11}}\dots b_{1\ell}^{\alpha_{1\ell}}+\dots+
a_k b_{k1}^{\alpha_{k1}}\dots b_{k\ell}^{\alpha_{k\ell}}\equiv c\pmod{m}
\end{equation}
has no solutions in non-negative integers $\alpha_{11},\dots,\alpha_{1\ell},\dots,
\alpha_{k1},\dots,\alpha_{k\ell}$.

\vskip.2truecm

\noindent
The conjecture is a variant of a classical conjecture of Skolem \cite{sk}. Note that the original formulation of Skolem is not completely precise; for an exact formulation one should e.g. see \cite{sch1}, pp. 398--399. If true, then the conjecture can be considered as a Hasse-type principle for exponential diophantine equations. There are several results in the literature about Skolem's conjecture; we only mention a theorem of Schinzel \cite{sch1} and a recent paper of Bartolome, Bilu and Luca \cite{bbl}, and the references given there.

The results of Schinzel \cite{sch1} also imply that in case of $k=1$ our conjecture is true. In this paper first we show that for any fixed $a_1,\dots,a_k$, $b_{11},\dots,b_{1\ell},\dots,b_{k1},\dots,b_{k\ell}$, the set of integers $c$ for which the above conjecture fails, has density zero even inside the set of those values $c$ for which equation \eqref{eq1} is not solvable. Moreover, here the appropriate moduli $m$ can be chosen to have the extra property that they are all divisible by $r$, for any preliminary chosen integer $r$. The main tools in the proof are a generalization of a classical result of Erd\H{o}s, Pomerance and Schmutz \cite{eps} concerning small values of Carmichael's $\lambda$-function, and a result of \'Ad\'am, Hajdu and Luca \cite{ahl} about the number of values $c$ up to any $x$, for which equation \eqref{eq1} is solvable. Further, we also give some "numerical evidence" for the conjecture, by checking its validity in different settings, and for a relatively large set of the parameters involved.

As an application, we present a general method for the solution of concrete equations of the type \eqref{eq1} under certain assumptions. Namely, if the Conjecture is true, then assuming that \eqref{eq1} has only finitely many solutions, our method makes it possible to find all these solutions, at least in principle. In fact the assumption about the finiteness of solutions can be relaxed. To illustrate the method, we present some concrete examples, as well. We mention that in the literature one can find several sparse results of this type. For example, Alex, Brenner and Foster in a series of papers (see e.g. \cite{bf,af1,af2} and the references there) solved several equations of type \eqref{eq1}, with typically $k=4,5$ and choices of $b_{11},\dots,b_{1\ell},\dots,b_{k1},\dots,b_{k\ell}$ as small primes. However, their way to find appropriate moduli like in \eqref{eq2} is rather ad-hoc, while in our method such moduli can be constructed systematically, based upon generalizations of arguments of Erd\H{o}s, Pomerance and Schmutz \cite{eps}.

Finally, we mention that we have implemented our algorithm in Sage \cite{sage}. The program, together with a complete description can be downloaded from the link www.math.unideb.hu/$\sim$hajdul/expeqsolver.zip.

\section{New results}

In our first result we show that the Conjecture formulated in the Introduction is true for "almost all" cases. For the precise formulation, we need the following notion. If $A\subseteq B\subseteq\mathbb Z$ and $B$, then the density of $A$ inside $B$ is defined as
$$
\lim\limits_{x\to\infty}\ {\frac{\#\{a\in A\ :\ |a|\leq x\}}{\#\{b\in B\ :\ |b|\leq x\}}},
$$
if the limit exists. Here and later on, $\# C$ denotes the number of elements of a set $C$.

\begin{theorem}
\label{thm1}
Let $a_1,\dots,a_k$ and $b_{11},\dots,b_{1\ell},\dots,b_{k1},\dots,b_{k\ell}$ be fixed, and let $H$ be the set of right hand sides in \eqref{eq1} for which the Conjecture is violated, that is
$$
H=\{c\in{\mathbb Z}\ :\ \eqref{eq1}\ \text{is not solvable, but}\ \eqref{eq2}\ \text{is solvable for all}\ m\}.
$$
Then $H$ has density zero inside the set
$$
H_0=\{c\in{\mathbb Z}\ :\ \eqref{eq1}\ \text{is not solvable}\}.
$$
\end{theorem}

Note that Theorem \ref{thm1} obviously implies that $H$ has density zero inside $\mathbb Z$.

In the proof of Theorem \ref{thm1} the following result plays an important role. This statement is a variant of a theorem of Erd\H{o}s, Pomerance and Schmutz \cite{eps} and Hajdu and Tijdeman \cite{ht1}. The important difference is the extra requirement that the appropriate moduli should be divisible by a fixed number $r$. This relation will play an important role in our method.

Let $\lambda(m)$ be the Carmichael function of the positive
integer $m$, that is the least positive integer for which
$$
b^{\lambda(m)}\equiv 1 \pmod{m}
$$
for all $b\in \mathbb{Z}$ with $\gcd(b,m)=1$. Later, we shall need the following information on small values of the Carmichael function.

\begin{theorem}
\label{thm2}
There exist positive constants $C_1>1$ and $C_2$ such that for any integer $r$ and for every large integer $i$ there is an integer $m$ with $r\mid m$, such that
$$
\log m
\in [\log i+\log r, (\log i)^{C_1}+\log r]
$$
and
$$
\lambda(m)<r(\log m/r)^{C_2\log\log\log m/r}.
$$
\end{theorem}

Our next theorem provides numerical evidence for the Conjecture, for various settings.

\begin{theorem}
\label{thm3}
Let $c$ be an integer with $0\leq c\leq 1000$. Then the Conjecture is valid for the following cases of equation \eqref{eq1}:
\begin{enumerate}
\item $p_1^{\alpha_1}-p_2^{\alpha_2}=c$ and $p_1^{\alpha_1}+p_2^{\alpha_2}-p_3^{\alpha_3}=c$ where $p_1,p_2,p_3$ are distinct primes less than $100$,
\item $p_1^{\alpha_1}+\dots+p_{t-1}^{\alpha_{t-1}}-p_t^{\alpha_t}=c$
where $p_1<\dots<p_t$ are primes less than $30$ with $4\leq t\leq 8$,
\item $p_1^{\alpha_1}p_2^{\alpha_2}+p_3^{\alpha_3}p_4^{\alpha_4}
-p_5^{\alpha_5}p_6^{\alpha_6}=c$ where $p_1,p_2,p_3,p_4,p_5,p_6$ are the primes $2,3,5,7,11,13$ in some order,
\item $2^{\alpha_1}+3^{\alpha_2}+5^{\alpha_3}+7^{\alpha_4}+11^{\alpha_5}
+13^{\alpha_6}+17^{\alpha_7}+19^{\alpha_8}-23^{\alpha_9}=55191$.
\end{enumerate}
\end{theorem}

\vskip.2truecm

\noindent
{\bf Remark.} The last equation in Theorem \ref{thm3} has no solutions, but it has solutions if $55191$ is replaced by any $c$ with $0\leq c<55191$.

\section{The application of the Conjecture to the explicit solution of exponential diophantine equations}

We propose the following principal strategy to find all solutions of equations of type \eqref{eq1}. For the moment, for simplicity assume that the equation has only finitely many solutions. We shall discuss the question that with what settings the strategy may work later.

\vskip.3truecm

\noindent
{\bf Principal strategy.}

\vskip.3truecm

\begin{itemize}

\item[(I)]{Find the suspected list of all solutions to equation \eqref{eq1} by an exhaustive search. [Note: Of course, at this point we cannot be sure that the list is complete. However, based upon the finiteness results concerning \eqref{eq1}, heuristically we may be strongly confident about it.]}
\item[(II)]{Choose one of the unknowns, $\alpha_{ij}$ say, and based upon the suspected list of all solutions take an integer $\alpha_0$ with $\alpha_{ij}<\alpha_0$. [Note: By choosing more than one unknowns we can speed up the calculations in an obvious way. However, to keep the presentation at this point simple, now we work only with one exponent.}
\item[(III)]{Instead of equation \eqref{eq1} consider the equation obtained by replacing the coefficient $a_i$ with $a_i b_{ij}^{\alpha_0}$. [Note: If our suspected list contained all solutions to \eqref{eq1} indeed, then the new equation has no solutions in non-negative integer exponents.]}
\item[(IV)]{Find an $m$ such that the new equation has no solution modulo $m$. Having such an $m$, conclude that $\alpha_{ij}<\alpha_0$ holds for all solutions of \eqref{eq1}.} [Note: If the Conjecture is true, then such a modulus exists. One can try to construct an appropriate $m$ by the help of Theorem \ref{thm2} (and its proof). Observe that for the unsolvability of the congruence modulo $m$ the relation $r:=b_{ij}^{\alpha_0}\mid m$ should hold, hence the importance of this property comes from.]
\end{itemize}

Observe that though the strategy contains heuristic points, once we succeed to find an appropriate modulus $m$ in the last step, it is justified that the original equation \eqref{eq1} has no solutions with $\alpha_{ij}\geq \alpha_0$. Hence we could get rid of an unknown, and we can repeat the whole procedure for an equation in one less variables than the original one. Finally, if everything works out well, we get all solutions.

This strategy works, at least in principle, if there exists a (not at all preliminary computable) constant $A$, such that for all solutions of \eqref{eq1} we have $\min\limits_{1\leq i\leq k,1\leq j\leq\ell} \alpha_{ij}<A$. (Since then we can eliminate one of the unknowns by the above method, etc.) This is the case, for example, if \eqref{eq1} has no solution with vanishing subsum.

At this point we mention that one can find in the literature several sparse results of this type; see e.g. the papers \cite{bf,af1,af2} and the references there. However, in these papers the appropriate moduli are found in a rather ad-hoc way, at least no clear strategy is explained to choose them. In our results we could use the moduli provided by Theorem \ref{thm2}. We give a detailed explanation in the proofs of our forthcoming theorems.

We illustrate our method by applying it to three branches of problems. In each case, we give two types of results. The first one always only shows that in the equations considered, one of the exponents can be bounded. To solve these equations completely one should iterate the method. The second type is where this iteration is executed, and the complete solution of a particular equation is presented.

Our next result concerns the representation of $c=0$ in \eqref{eq1} as sums and differences of powers of several distinct primes. Note that this result is closely related to a question of Brenner and Foster \cite{bf}.

\begin{theorem}
\label{thm4}
\hspace{1mm}
\begin{enumerate}
\item Let $3\le t\le 6$ and let $p_1,\ldots,p_t$ be distinct primes with $p_i\le 19$ $(i=1,\ldots,t)$. Then for the non-negative integer solutions $\alpha_1,\dots,\alpha_t$ of the equation
$$
p_1^{\alpha_1}+\dots+p_{t-1}^{\alpha_{t-1}}-p_t^{\alpha_t}=0
$$
we have $\min\limits_{1\leq i\leq t}\alpha_i\le 15$.
\item The equation
$$
3^{\alpha_1}+5^{\alpha_2}+11^{\alpha_3}+13^{\alpha_4}+ 17^{\alpha_5}-{19^{\alpha_6}}=0
$$
has only two solutions in non-negative integers $\alpha_1,\dots,\alpha_6$, given by
$$
(\alpha_1,\alpha_2,\alpha_3,\alpha_4,\alpha_5,\alpha_6)
=(0,1,1,0,0,1),(1,0,0,1,0,1).
$$
\end{enumerate}
\end{theorem}

The following theorem concerns the case where all but one primes are equal. Obviously, in this case the exponents of these primes can be arranged in a non-decreasing way. Further, the smallest exponent must always be zero, that is why the constant $1$ appears on the left hand side.

\begin{theorem}
\label{thm5}
\hspace{1mm}
\begin{enumerate}
\item Let $3\le t\le 9$ and let $p,q$ be distinct primes with $p,q\le 19$. Then for the non-negative integer solutions $\alpha_1,\dots,\alpha_t$ of the equation
$$
1+p^{\alpha_1}+\dots+p^{\alpha_{t-1}}-q^{\alpha_t}=0
$$
we have $\min\limits_{1\leq i\leq t}\alpha_i\le 6$.
\item The diophantine equation
$$
1+5^{\alpha_1}+5^{\alpha_2}+5^{\alpha_3}+5^{\alpha_4} +5^{\alpha_5}+5^{\alpha_6}+5^{\alpha_7}+ 5^{\alpha_8}-17^{\alpha_9}=0
$$
has only two solutions in non-negative integers $\alpha_1,\dots,\alpha_9$ with $\alpha_1\leq\dots\leq \alpha_8$, given by
$$
(\alpha_1,\alpha_2,\alpha_3,\alpha_4,\alpha_5, \alpha_6,\alpha_7,\alpha_8,\alpha_9)=
$$
$$
=(0,0,0,0,0,0,1,1,1),(0,0,0,1,1,2,3,3,2).
$$
\end{enumerate}
\end{theorem}

Our final result concerns the case $\ell=2$ in \eqref{eq1}.

\begin{theorem}
\label{thm6}
\hspace{1mm}
\begin{enumerate}
\item Let $p_1,\ldots,p_6$ be distinct primes with $p_i\le 19$ $(i=1,\ldots,6)$. Then for the non-negative integer solutions $\alpha_1,\dots,\alpha_6$ of the equation
$$
p_1^{\alpha_1}p_2^{\alpha_2} + p_3^{\alpha_3}p_4^{\alpha_4} - p_5^{\alpha_5}p_6^{\alpha_6} =1
$$
we have $\min\limits_{1\leq i\leq 6}\alpha_i\le 5$.
\item The equation
$$
2^{\alpha_1}3^{\alpha_2} + 5^{\alpha_3}7^{\alpha_4} - 11^{\alpha_5}13^{\alpha_6} =1
$$
has only two solutions in non-negative integers $\alpha_1,\dots,\alpha_6$, given by
$$
(\alpha_1,\alpha_2,\alpha_3,\alpha_4,\alpha_5,\alpha_6)
=(0,0,0,0,0,0), (0,2,1,0,0,1).
$$
\end{enumerate}
\end{theorem}

\section{Proofs}

We start with the proof of Theorem \ref{thm2}.

\begin{proof}[Proof of Theorem \ref{thm2}]
Theorem 5 of \cite{ht1} is just the statement with $r=1$. So let $C_1$ and $C_2$ be the constants implied by Theorem 5 of \cite{ht1}, let $r$ be an arbitrary positive integer, and let $i$ be sufficiently large. Then by Theorem 5 of \cite{ht1} there exists an $n$ such that
$$
\log n
\in [\log i, (\log i)^{C_1}]\ \text{and}\
\lambda(n)<(\log n)^{C_2\log\log\log n}.
$$
Put $m:=rn$. Then obviously, $r\mid m$. Further, we immediately obtain
$$
\log m\in [\log i+\log r, (\log i)^{C_1}+\log r].
$$
Finally, as it is well-known, for any positive integers $a,b$ we have $\lambda(ab)\leq a\lambda(b)$. Hence
$$
\lambda(m)\leq r\lambda(n)<r(\log n)^{C_2\log\log\log n}=
r(\log m/r)^{C_2\log\log\log m/r},
$$
and the theorem follows.
\end{proof}

Now we continue with the proof of Theorem \ref{thm1}. For this, beside Theorem \ref{thm2} we need the following results of \'Ad\'am, Hajdu and Luca \cite{ahl}.

\begin{lemma}
\label{lem1}
Using the notation of Theorem \ref{thm1}, write $H_0(x)$ for the elements $h$ of $H_0$ with $|h|\leq x$ where $x$ is a positive real number. Then for all large $x$ we have
$$
\# H_0(x)>2x-C_3(\log x)^{C_4}
$$
where $C_3$ and $C_4$ are constants depending only on the parameters $k$, $a_i$ and $b_{ij}$ occurring in \eqref{eq1}.
\end{lemma}

\begin{proof} The statement is a simple consequence of Theorem 1 of \cite{ahl}.
\end{proof}

We also need the following

\begin{lemma}
\label{lem2} Let $m=q_1^{\beta_1}\cdots q_z^{\beta_z}$ where
$q_1,\ldots,q_z$ are distinct primes, $\beta_1,\dots,\beta_z$ are positive integers, and let $b\in\mathbb Z$. Then we have
$$
\#\{b^u \pmod m : u\geq 0\}\le \lambda(m)+ \max\limits_{1\le i\le z}\alpha_i.
$$
\end{lemma}

\begin{proof} The statement is Lemma 1 in \cite{ahl}.
\end{proof}

Now we are ready to prove Theorem \ref{thm1}.

\begin{proof}[Proof of Theorem \ref{thm1}] For a positive real number $x$ set
$$
H(x):=\{h\in H:\ |h|\leq x\}\ \ \
\text{and}\ \ \
H_0(x):=\{h\in H_0:\ |h|\leq x\}.
$$
We apply Lemmas \ref{lem1} and \ref{lem2}, and Theorem \ref{thm2} with $r=1$ to prove our statement. Partly we follow the argument of Theorem 1 of \cite{ht2}; see also the proof of Theorem 3 in \cite{ahl}.

Throughout the proof, we assume that $x$ is large enough for the arguments to hold. By Theorem \ref{thm2} we can choose an integer $m$, satisfying $m\leq \sqrt{x}$ and
\begin{equation}
\label{neweq}
\lambda(m) < (\log m)^{C_2\log\log\log m}.
\end{equation}
We may assume that $m$ is the largest integer with these properties. Then by Theorem \ref{thm2} we have that $m>f(x)$, with some monotone increasing function $f$ of $x$, tending to infinity as $x$ goes to infinity.

Let $m=q_1^{\beta_1}\cdots q_z^{\beta_z}$ be the prime factorization of $m$, where $q_1,\ldots,q_z$ are distinct primes and $\beta_1,\ldots,\beta_z$ are positive integers.
Write $C(m)$ for the collection of the modulo $m$ residue classes of those integers $c$ for which the congruence \eqref{eq2} is solvable.
Lemma \ref{lem2} implies that we have
\begin{equation}
\label{ineq1} \# C(m) \le
(\lambda(m)+\max\limits_{1\le i\le z}\beta_i)^{k\ell}.
\end{equation}
On the other hand, by \eqref{neweq} we easily get that
\begin{equation}
\label{ineq2} \lambda(m)+\max\limits_{1\le i\le z}\beta_i\le (\log m)^{C_2\log\log\log m}+\frac{\log
m}{\log 2}.
\end{equation}
Now by inequalities \eqref{ineq1} and \eqref{ineq2} we get that
\begin{equation}
\label{ineq3}
\# C(m)<(\log m)^{C_5\log\log\log m},
\end{equation}
where $C_5$ is a constant depending only on $k$ and $\ell$.

Write now $x=um+v$ where $u$ is a positive integer and $v$ is a non-negative real number with $v<m$. Observe that by our choice of $m$, $u$ and $v$ we have that
$$
\log u\geq (\log (u+1))/2\geq (\log x-\log m)/2\geq (\log x)/4.
$$
Let now $\varepsilon$ be an arbitrary positive real number.
Then the above inequality implies
$$
\varepsilon x/3 \geq\varepsilon um/3>m>v.
$$
Further, we also have
$$
\varepsilon x/3>C_3(\log x)^{C_4}/2,
$$
where $C_3$ and $C_4$ are given in Lemma \ref{lem1}. Finally, the lower bound $m>f(x)$ also gives
$$
\varepsilon x/3\geq \varepsilon um/3>u(\log m)^{C_5\log\log\log m}.
$$
Thus, since by \eqref{ineq3} and $x=um+v$ we have that
$$
\# H(x)\leq 2(u(\log m)^{C_5\log\log\log m}+v)+1,
$$
the statement immediately follows by comparing the above inequality with
$$
\# H_0(x)>2x-C_3(\log x)^{C_4},
$$
given by Lemma \ref{lem1}.
\end{proof}

\begin{proof}[Proof of Theorem \ref{thm3}] Since the proofs of the parts (1) to (4) are similar, we only give details in case of (3). Moreover, here we consider only the equations
\begin{equation}
\label{spec3}
2^{\alpha_1}3^{\alpha_2}+5^{\alpha_3}7^{\alpha_4} -11^{\alpha_5}13^{\alpha_6}=c
\end{equation}
with $0\leq c\leq 1000$. First, letting the exponents $\alpha_i$ $(i=1,\dots,6)$ vary between $0$ and $12$, we find a list $L$ (with $\# L=224$) of $c$ values for which we expect equation \eqref{spec3} not to have solutions. (Note that some equations in (3) have solutions with $\max\limits_{1\leq i\leq 6} \alpha_i=12$.) At this stage, at least we are certain that for integers $c$ with $0\leq c\leq 1000$ not in the list $L$, equation \eqref{spec3} is solvable. Now we investigate the values $c\in L$ one by one. The smallest such value is $c=11$, we shall work only with this, the others can be handled similarly. Take the modulus
$$
m:=7031324575728=2^4\cdot 3^2\cdot 17\cdot 19\cdot 37\cdot 73\cdot 97\cdot 577.
$$
(Later we shall explain how to find this $m$.) Now we could simply say that as one can easily check, equation \eqref{spec3} has no solutions modulo $m$. However, as this check is not that easy for some of the instances in (1) to (4), it is worth to do it in a sophisticated way. (In particular, since the appropriate modulus $m$ can be much larger than the one given above.)

First observe that all the factors of $m$ have $\lambda$ values composed exclusively of $2$-s and $3$-s. (This is the choice indicated by the proof of Erd\H{o}s, Pomerance and Schmutz \cite{eps}.) This makes it possible to combine the information obtained for the coefficients $\alpha_1,\dots,\alpha_6$ modulo the separate factors. (It is highly not economic to work with $m$ as a modulus directly.) For example, modulo $2^4$ we immediately get that $\alpha_1=0$ must hold, and we also get some congruence conditions for the other exponents, modulo a power of $2$ (since the orders of all the factors modulo $2^4$ are certainly powers of $2$). Then, modulo $3^2$ we get further conditions on $\alpha_3$, $\alpha_4$, $\alpha_5$ and $\alpha_6$, modulo ord$_9(5)=6$, ord$_9(7)=2$, ord$_9(11)=6$ and ord$_9(13)=3$, respectively. Finally, using all the factors of $m$ as modulus, the resulting system of congruences obtained for the exponents $\alpha_1,\dots,\alpha_6$ proves to be non-solvable. This shows that equation \eqref{spec3} with $c=11$ has no solutions modulo $m$ indeed.

In all the other cases the proof goes along the same lines. In some cases one really needs to work with huge moduli. However, in all cases we encountered, the modulus
$$
m^*=2^4\cdot 3^2\cdot {\underset{3<p<20000}{\prod\limits_{p-1=2^u3^v5^w} p}}
$$
proved to be appropriate. That is, the $m$ we found was always a divisor of $m^*$.

Finally, we explain how we found the appropriate moduli $m$. In fact the outlined procedure in many cases could be simplified, e.g. starting with a shorter list of prime powers.

Let $M$ be the list of all prime power divisors of $m^*$. Consider an equation of the form
\begin{equation}
\label{newneweq}
b_1^{\alpha_1}b_2^{\alpha_2}+b_3^{\alpha_3}b_4^{\alpha_4} -b_5^{\alpha_5}b_6^{\alpha_6}=c
\end{equation}
(in all the other instances the procedure is similar).
Define a heuristic measurement $f(t)$ for the "goodness" of the elements $t$ of $M$, with respect to the bases $b_1,\dots,b_6$. We take the function $f(t)$ defined as
$$
f(t)=o_1o_2o_3o_4o_5o_6
$$
where $o_i=\text{ord}_t(b_i)$ $(i=1,\dots,6)$, with the convention $o_i=\text{ord}_t(b_i)=1$ if $\gcd(t,b_i)>1$.
Then we take the first $t$ from $M$ as modulus, for which $f(t)$ is minimal. By this modulus, we obtain some conditions for the exponents $\alpha_i$ modulo $o_i$ $(i=1,\dots,6)$. In particular, if $t$ is a power of the prime $b_i$, then we know that either $\alpha_i$ is smaller than the exponent of $b_i$ in $t$, or $b_i^{\alpha_i}\equiv 0\pmod{t}$.

For simplicity, suppose that $\gcd(b_i,t)=1$ $(i=1,\dots,6)$ and that we have
$$
\alpha_i\equiv \beta_i\pmod{o_i}\ \ \ (i=1,\dots,6),
$$
with some $\beta_i$ subject to $0\leq\beta_i<o_i$. Then we can rewrite equation \ref{newneweq} as
$$
a_1(b_1')^{\gamma_1}(b_2')^{\gamma_2}+ a_2(b_3')^{\gamma_3}(b_4')^{\gamma_4} -a_3(b_5')^{\gamma_5}(b_6')^{\gamma_6}=c
$$
with
$$
a_1=b_1^{\beta_1}b_2^{\beta_2},\ a_2=b_3^{\beta_1}b_4^{\beta_2},\ a_3=b_5^{\beta_1}b_6^{\beta_2},
$$
and
$$
b_i'=b_i^{o_i},\ \gamma_i=(\alpha_i-\beta_i)/o_i\ (i=1,\dots,6).
$$
Now we can apply the above method for this equation with $M$ replaced by $M\setminus \{t\}$, etc. In this way we could always guarantee that the next modulus $t$ is the actually "best", which makes the computation relatively fast.

Note that all the necessary exponentiations can be made locally, which keeps the procedure economic. The calculations have been performed by the program package Sage \cite{sage}.
\end{proof}

\begin{proof}[Proof of Theorem \ref{thm4}] We only deal with the second statement, since it asserts the complete solution of an equation. Deriving an upper bound for one of the exponents will be a part of our method, so part (1) of the theorem can be proved in a similar way.

First, according to (I) of our Principal strategy, we find a suspected list of all solutions of the equation
\begin{equation}
\label{mintaegy}
3^{\alpha_1}+5^{\alpha_2}+11^{\alpha_3}+13^{\alpha_4}+ 17^{\alpha_5}-19^{\alpha_6}=0
\end{equation}
with $0\leq \alpha_1,\ldots,\alpha_6<15$. We get only two solutions, namely
$$
(\alpha_1,\alpha_2,\alpha_3,\alpha_4,\alpha_5,\alpha_6) =(0,1,1,0,0,1),(1,0,0,1,0,1).
$$
So we strongly suspect that there are no other solutions.
Now, following steps (II) and (III) of the strategy, but shifting the powers of all bases, we consider the equation
$$
3^2\cdot 3^{\alpha_1'} + 5^2\cdot 5^{\alpha_2'} + 11^2\cdot 11^{\alpha_3'} + 13^2\cdot 13^{\alpha_4'} + 17\cdot 17^{\alpha'_5} - 19^2\cdot 19^{\alpha_6'}=0.
$$
If our list of two solutions is complete, then the above equation has no solutions in non-negative integers $\alpha_1',\dots,\alpha_6'$. To show this, we find a modulus $m$ such that the congruence
$$
3^2\cdot 3^{\alpha_1'} + 5^2\cdot 5^{\alpha_2'} + 11^2\cdot 11^{\alpha_3'} + 13^2\cdot 13^{\alpha_4'} + 17\cdot 17^{\alpha'_5} - 19^2\cdot 19^{\alpha_6'}\equiv 0\pmod{m}
$$
has no solutions. By a similar strategy as in the proof of Theorem \ref{thm3}, we find that
$$
m=2\cdot 3^2\cdot 5\cdot 7\cdot 13\cdot 17\cdot 19\cdot 37\cdot 73\cdot 109\cdot 163\cdot 433
$$
is an appropriate modulus. This means that in any solution of \eqref{mintaegy}, one of
$$
\alpha_1\leq 1,\ \ \alpha_2\leq 1,\ \ \alpha_3\leq 1,\ \ \alpha_4\leq 1,\ \ \alpha_5=0,\ \ \alpha_6\leq 1
$$
must be valid. This means that one of the exponents $\alpha_1,\dots,\alpha_6$ is fixed, and may take at most two values. We consider only the possibility $\alpha_1=0$, the other cases can be treated similarly. In this case our equation reads as
$$
1+5^{\alpha_2}+11^{\alpha_3}+13^{\alpha_4}+17^{\alpha_5} -19^{\alpha_6}=0.
$$
Based upon our previous calculations, we suspect that this new equation has the only solution
$$
(\alpha_2,\alpha_3,\alpha_4,\alpha_5,\alpha_6)=(1,1,0,0,1).
$$
Now similarly as above, but shifting only the exponent of $13$, if the conjecture is true, then there exists a modulus $m$ such that the congruence
$$
1+5^{\alpha_2}+11^{\alpha_3}+13\cdot 13^{\alpha_4'}+17^{\alpha_5} -19^{\alpha_6}\equiv 0\pmod{m}
$$
has no solutions modulo $m$. Now by the same method as previously, we get that the modulus
$$
m=2^4\cdot 3^2\cdot 7\cdot 13\cdot 37\cdot 73\cdot 109\cdot 433
$$
verifies this assertion. Hence we obtain that $\alpha_4=0$ must be valid, and our equation reduces to
$$
2+5^{\alpha_2}+11^{\alpha_3}+17^{\alpha_5}-19^{\alpha_6}=0.
$$

By our earlier calculations, we strongly suspect that this equation has the only solution
$$
(\alpha_2,\alpha_3,\alpha_5,\alpha_6)=(1,1,0,1).
$$
Shifting now the exponent of $17$ by one, the modulus
$$
m=2\cdot3\cdot7\cdot17\cdot37\cdot73\cdot97\cdot109\cdot163
$$
witnesses that the equation arising has no solutions - in other words, we must have $\alpha_5=0$.

Hence we obtain the equation
$$
3+5^{\alpha_2}+11^{\alpha_3}-19^{\alpha_6}=0,
$$
with the only expected solution
$$
(\alpha_2,\alpha_3,\alpha_6)=(1,1,1).
$$
Now we shift the exponent of $5$ to get the equation
$$
3+5^2\cdot 5^{\alpha_2'}+11^{\alpha_3}-19^{\alpha_6}=0,
$$
which turns out to have no solutions modulo
$$
m=3^3\cdot 5^2\cdot 7\cdot 31.
$$
This leaves us with the equations
$$
4+11^{\alpha_3}-19^{\alpha_6}=0\ \ \text{and}\ \
8+11^{\alpha_3}-19^{\alpha_6}=0.
$$
The first equation has no solutions modulo $3$. Further, we suspect that the only solution of the second equation is
$$
(\alpha_3,\alpha_6)=(1,1).
$$
Now we shift the exponent of $11$ to obtain the equation
$$
8+11^2\cdot 11^{\alpha_3'}-19^{\alpha_6}=0,
$$
which has no solutions modulo
$$
m=5^2\cdot 11^2 \cdot 31\cdot 61.
$$
This means that $\alpha_3=0$ or $\alpha_3=1$. From this we easily get that $\alpha_3=\alpha_6=1$ must be valid. By following similar arguments, we could solve all the encountered equations and we get that the solutions are those listed in the statement.
\end{proof}

\begin{proof}[Proof of Theorem \ref{thm5}] The proof of this theorem is very similar to that of Theorem \ref{thm4}, so we only indicate the main steps. Again, we only deal with part (2) of the statement, part (1) could be handled similarly.

After finding the suspected solutions
$$
(\alpha_1,\alpha_2,\alpha_3,\alpha_4,\alpha_5, \alpha_6,\alpha_7,\alpha_8,\alpha_9)=
$$
$$
=(0,0,0,0,0,0,1,1,1),(0,0,0,1,1,2,3,3,2),
$$
using the modulus
$$
m=2\cdot 3\cdot 5^2\cdot 7\cdot 13\cdot 31\cdot 601
$$
we successively get $\alpha_1=\alpha_2=\alpha_3=0$ and $\alpha_4\leq 1$.

In case of $\alpha_4=0$, using again $m$ we successively obtain $\alpha_5=\alpha_6=0$ and $\alpha_7=\alpha_8=1$, whence $\alpha_9=1$, and we obtain the first solution calculated preliminary. Note that in some of the above arguments, $m$ could be replaced by $m/5$ or $m/601$.

When $\alpha_4=1$, a similar calculation (but with rather more complicated moduli) leads to the second solution, and the theorem follows.
\end{proof}

\begin{proof}[Proof of Theorem \ref{thm6}] Again, we only deal with part (2) of the statement. Since the proof is very similar to the previous ones, we only give the moduli used, and the information deduced for the exponents.

We start with the equation
$$
2^{\alpha_1}3^{\alpha_2} + 5^{\alpha_3}7^{\alpha_4} - 11^{\alpha_5}13^{\alpha_6}=1.
$$
Then modulo $2$ we get that $\alpha_1=0$. So the equation reduces to
$$
3^{\alpha_2} + 5^{\alpha_3}7^{\alpha_4} - 11^{\alpha_5}13^{\alpha_6}=1.
$$
Now taking
$$
m=2^5\cdot 7\cdot 17\cdot 19\cdot 37\cdot 73\cdot 97\cdot 193
$$
we obtain that $\alpha_4=0$. Then our equation takes the form
$$
3^{\alpha_2} + 5^{\alpha_3} - 11^{\alpha_5}13^{\alpha_6}=1.
$$
Taking
$$
m=7\cdot 11\cdot 17\cdot 19\cdot 31\cdot 37\cdot 41\cdot 73\cdot 97\cdot 193
$$
we get that $\alpha_5=0$, that is, we have to solve
$$
3^{\alpha_2} + 5^{\alpha_3} - 13^{\alpha_6}=1.
$$
This equation modulo
$$
m=5^2\cdot 7\cdot 11\cdot 31\cdot 41
$$
yields that $\alpha_3\leq 1$. The equality $\alpha_3=0$ trivially leads to $\alpha_2=\alpha_6=0$. This gives the first solution. So we are left with the case $\alpha_3=1$, when the equation is of the shape
$$
3^{\alpha_2}-13^{\alpha_6}=-4.
$$
Then modulo
$$
m=27\cdot 7\cdot 19\cdot 37
$$
we get that $\alpha_2\leq 2$, which easily yields $\alpha_2=2$ and $\alpha_6=1$. Thus we get the second solution, and the theorem follows.
\end{proof}

\end{document}